\documentclass[12pt]{article}
\usepackage{amssymb,amsthm,amsmath,latexsym}
\usepackage{authblk}
\usepackage[symbol]{footmisc}
\usepackage{graphics}
\usepackage{color}
\usepackage{graphicx}

\newtheorem {theorem}{Theorem}

\newtheorem {definition}{Definition}
\newtheorem {proposition}{Proposition}
\newtheorem {remark}{Remark}
\newtheorem {example}{Example}

\begin{document}
\title{The $r$-central factorial numbers with even indices }
\author{F. A. Shiha}
\affil{ Department of Mathematics, Faculty of Science, Mansoura University, 35516 Mansoura, Egypt}
\date{}
\maketitle
\begin{abstract}
  In this paper, we introduce the $r$-central factorial numbers with even indices of the first and second kind,
  as extended versions of the central factorial numbers with even indices of both kinds. We obtain several fundamental properties and identities related to these numbers. We show that the unsigned $r$-central factorial numbers with even indices of the first kind are strictly log-concave and Poisson-binomially distributed . Finally, we consider the $r$-central factorial matrices and the factorization of it.

 \end{abstract}
\maketitle

\bigskip AMS (2010) Subject Classification: 05A15, 11B75, 11B83, 60C05.


\section {Introduction}
The central factorial numbers of the first kind $t(n,k)$ and of the second kind $T(n,k)$, which are defined in Riodan \cite[page 213-217]{riod} by
\begin{equation}\label{E:cent1}
x^{[n]}=\sum_{k=0}^{n}t(n,k)x^k,
\end{equation}
\begin{equation}\label{E:cent2}
 x^n=\sum_{k=0}^{n}T(n,k)x^{[k]},
\end{equation}
where $x^{[n]}=x(x+\frac{n}{2}-1)(x+\frac{n}{2}-2)\cdots (x-\frac{n}{2}+1)$, $n \geq 1$, $x^{[0]}=1$.

These numbers satisfy the following recurrence relations:
\[
t(n,k)=t(n-2,k-2)-\frac{1}{4}(n-2)^2\: t(n-2,k),\quad 2\leq k\leq n,
\]
\[
T(n,k)=T(n-2,k-2)+\frac{1}{4}k^2\: T(n-2,k),\quad 2\leq k\leq n,
\]
with $t(n,0)=T(n,0)=\delta_{n0}$.

The generating functions of these numbers are given by (see \cite{kim119})
\[
\sum_{n=k}^{\infty}t(n,k)\frac{t^n}{n!}=\frac{1}{k!}\left(2\log\left(\frac{t}{2}+\sqrt{1+\frac{t^2}{4}}\right)\right)^k,
\]
\[
\sum_{n=k}^{\infty}T(n,k)\frac{t^n}{n!}=\frac{1}{k!}\left(e^{\frac{t}{2}}-e^{\frac{-t}{2}}\right)^k.
\]
Many properties of the central factorial numbers can be found in Butzer et al. \cite{but}.

For any nonnegative integer $r$, Kim et al. \cite{kim119} generalized the central factorial numbers of the second kind  to the extended $r$-central factorial numbers of the second kind $T_r(n + r, k + r)$ by

\[
\frac{1}{k!}e^{rt}\left(e^{\frac{t}{2}}-e^{\frac{-t}{2}}\right)^k=\sum_{n=k}^{\infty}T_r(n+r,k+r)\frac{t^n}{n!}.
\]
Kim et al. \cite{kim19} defined  also the extended central factorial numbers of the second kind by
\[
\frac{1}{k!}\left(e^{\frac{t}{2}}-e^{\frac{-t}{2}}+rt\right)^k=\sum_{n=k}^{\infty}T^{(r)}(n,k)\frac{t^n}{n!}.
\]
Note that the numbers $T_r(n+r,k+r)$ and $T^{(r)}(n,k)$ are different.

In \cite{gel}, the central factorial numbers with even indices of the first and second kind, respectively, are denoted by
\begin{equation}
 u(n,k)=t(2n,2k) \quad \text{and} \quad U(n,k)=T(2n,2k).
\end{equation}
They satisfy the recurrence relations:
\begin{equation}
  u(n,k)=u(n-1,k-1)-{(n-1)}^2 u(n-1,k),
\end{equation}
\begin{equation}
  U(n,k)=U(n-1,k-1)+k^2 U(n-1,k).
\end{equation}
The explicit formula and the exponential generating function of $U(n,k)$ are given by
\begin{equation}
 U(n,k)= \frac{2}{(2k)!}\sum_{j=1}^{k}(-1)^{k+j}\binom{2k}{k-j}j^{2n},
\end{equation}
\begin{equation}\label{E:ecent2}
  \sum_{n=k}^{\infty} U(n,k)\frac{t^n}{n!}= \frac{2}{(2k)!}\sum_{j=0}^{k}(-1)^{k+j}\binom{2k}{k-j}e^{j^2t}.
\end{equation}
The combinatorial interpretations of $u(n,k)$ and $U(n,k)$ can be found in \cite{gel}, and the connections between these numbers and Bernoulli polynomials found in \cite{merca2016}.
\section{The $r$-central factorial numbers with even indices}
Here, we introduce the $r$-central factorial numbers with even indices of the first and second kind, which we denote by $u_r(n,k)$ and $U_r(n,k)$, respectively.
\begin{definition}The arrays $u_r(n,k)$ and $U_r(n,k)$ for non-negative integers $r, n$ and $ k $ with $n \geq k$ determined by the recurrences
\begin{equation}\label{re1}
  u_r(n,k)=u_r(n-1,k-1)-((n-1)^2+r)u_r(n-1,k),\quad n\geq k\geq 1,
\end{equation}
\begin{equation}\label{re2}
  U_r(n,k)=U_r(n-1,k-1)+(k^2+r)U_r(n-1,k), \quad n\geq k\geq 1,
\end{equation}
with initial values $ u_r(n,0)= (-1)^n \prod_{i=0}^{n-1}(i^2+r)$, $U_r(n,0)=r^n$ and $u_r(0,k)=U_r(0,k)=\delta_{k0}$, for $n, k \geq 0$.
\end{definition}
Note that at $r=0$, these numbers are reduced to the central factorial numbers with even indices, $u_0(n,k)=u(n,k)$ and $U_0(n,k)=U(n,k)$.

From (\ref{re1}) and (\ref{re2}), it is easy to observe that
\begin{align*}
&u_r(n,1)=(-1)^{n-1} \prod_{l=0}^{n-1}(l^2+r) \sum_{i=0}^{n-1}  \frac{1}{r+i^2}, &U_r(n,1)&=(r+1)^n-r^n, \\
&u_r(n,n-1)=-\sum_{l=0}^{n-1} ( r+l^2), & U_r(n,n-1)&=\sum_{l=0}^{n-1} ( r+l^2),\\
&u_r(n,n)=1,  &U_r(n,n)&=1.
\end{align*}
In the next, we get an explicit formula for the array $U_r(n,k)$ and further recurrences satisfied by $u_r(n,k)$ and $U_r(n,k)$ by using the following theorem.
\begin{theorem} \textrm{(Mansour et al. \cite{mansour12})}.
 Suppose that the array $y(n,k)$, $n,k \geq 0$ is defined by
\begin{equation}
y(n,k)=y(n-1,k-1)+(a_{n-1}+b_k)y(n-1,k), \qquad n,k \geq 1
\end{equation}
 with $y(n,0)=\prod_{i=0}^{n-1}(a_i+b_0)$ and $y(0,k)=\delta_{0k}$, for all $n,k \geq 0$, where
 $\{a_i\}_{i\geq 0}$ and $\{b_i\}_{i\geq 0}$ are given sequences  with the  $b_i$ distinct, then
\begin{equation} \label{E:mex}
 y(n,k)=\sum_{j=0}^k \left( \frac{\prod_{i=0}^{n-1}(b_j+a_i)}{\prod_{i=0 ,i\ne j}^k(b_j-b_i)}\right),\: \:\:\forall\: n,k \in \mathbb{N},
 \end{equation}
 and
 \begin{equation}\label{E:mrec}
y(n,k)=\sum_{j=k}^n y(j-1,k-1)\prod_{i=j}^{n-1}(a_i+b_k).
\end{equation}

\end{theorem}
\begin{theorem}For any integer $0\leq k \leq n$,
 \begin{equation}\label{E:rcent22}
  U_r(n,k)=\frac{2}{(2k)!}\sum_{j=0}^{k}(-1)^{k+j}\binom{2k}{k-j}(j^2+r)^n.
 \end{equation}
\begin{equation}\label{rcent111}
  u_r(n,k)= \sum_{l=k}^{n}(-1)^{n-l}u_r(l-1,k-1)\prod_{i=l}^{n-1} (i^2+r).
\end{equation}
\begin{equation}\label{rcent222}
  U_r(n,k) =\sum_{l=k}^{n} U_r(l-1,k-1)(k^2+r)^{n-l}.
  \end{equation}
\end{theorem}
 \begin{proof}
Taking $a_j= 0 $ and $b_j=j^2+r$ for all $j$ in (\ref{E:mex}), then
\begin{equation*}
 \begin{split}
U_r(n,k)&=\sum_{j=0}^{k}\frac{\prod_{i=0}^{n-1}(j^2+r)}{\prod_{i=0 ,i\ne j}^{k}(j^2-i^2)} =\sum_{j=0}^{k}\frac{2(-1)^{k+j}(j^2+r)^n}{(k-j)!\:(k+j)!} \\
&=  \frac{2}{(2k)!}\sum_{j=0}^{k}(-1)^{k+j}\binom{2k}{k-j}(j^2+r)^n.
\end{split}
\end{equation*}
For (\ref{rcent111}), set $a_i=-(i^2+r)$, $b_k=0$ in (\ref{E:mrec}), and for (\ref{rcent222}), set $a_i=0$, $b_k=k^2+r$ in  (\ref{E:mrec}).
\end{proof}
Multiplying both sides of (\ref{E:rcent22}) by $\frac{t^n}{n!}$ and summing over $n \geq k$ gives the exponential generating function of $U_r(n,k)$,
\begin{equation}\label{E:ercent2}
 \sum_{n=k}^{\infty} U_r(n,k)\frac{t^n}{n!}= \frac{2}{(2k)!}\sum_{j=0}^{k}(-1)^{k+j}\binom{2k}{k-j}e^{(j^2+r)t}.
\end{equation}
The following theorem shows how to express $U_r(n,k)$ in terms of $U(n,k)$.
\begin{theorem}
For $n, k, r \geq 0$ with $n \geq k$,
\begin{equation}\label{ev2}
  U_r(n,k)=\sum_{l=k}^{n}\binom{n}{l}U(l,k)r^{n-l}.
\end{equation}
\end{theorem}
\begin{proof}
  By (\ref{E:ecent2}) and (\ref{E:ercent2}),
 \begin{equation*}
\begin{split}
\sum_{n=k}^{\infty} U_r(n,k)\frac{t^n}{n!}&= \frac{2}{(2k)!}\sum_{j=0}^{k}(-1)^{k+j}\binom{2k}{k-j}e^{j^2t}\:e^{rt}\\
&=\left(\sum_{l=k}^{\infty}U(l,k)\frac{t^l}{l!}\right)\left(\sum_{m=0}^{\infty}\frac{r^mt^m}{m!}\right)\\
&=\sum_{n=k}^{\infty}\left(\sum_{l=k}^{n}\binom{n}{l}U(l,k)r^{n-l}\right)\frac{t^n}{n!}.\\
\end{split}
\end{equation*}
Hence, the identity holds by comparing the coefficients on both sides.
\end{proof}
We next show that $u_r(n,k)$ and $U_r(n,k)$ can be defined as connection coefficients between
some special polynomials.
\begin{theorem} For $n\geq 0$, then
\begin{equation}\label{E:rcent1}
\prod_{i=0}^{n-1}(x-i^2)=\sum_{k=0}^{n}u_r(n,k)(x+r)^k,
\end{equation}
\begin{equation}\label{E:rcent2}
(x+r)^n=\sum_{k=0}^{n}U_r(n,k)\prod_{i=0}^{k-1}(x-i^2).
\end{equation}
\end{theorem}
\begin{proof}
We prove (\ref{E:rcent1}) by induction on $n$ and (\ref{E:rcent2}) is proven similarly.
The initial case of $n=0,1$ being obvious. Suppose that the statement is true for $n$, we prove it for $n+1$.
\begin{equation*}
\begin{split}
&\sum_{k=0}^{n+1} u_r(n+1,k)(x+r)^k=\sum_{k=0}^{n}u_r(n+1,k)+(x+r)^{n+1}\\
&=\sum_{k=0}^{n}u_r(n,k-1)(x+r)^k-(n^2+r)\sum_{k=0}^{n}u_r(n,k)(x+r)^k+(x+r)^{n+1}\\
&=\sum_{k=0}^{n-1}u_r(n,k)(x+r)^{k+1}-(n^2+r)\prod_{i=0}^{n-1}(x-i^2)+(x+r)^{n+1}\\
&=\sum_{k=0}^{n}u_r(n,k)(x+r)^{k+1}-(x+r)^{n+1}-(n^2+r)\prod_{i=0}^{n-1}(x-i^2)+(x+r)^{n+1}\\
&=(x+r)\prod_{i=0}^{n-1}(x-i^2)-(n^2+r)\prod_{i=0}^{n-1}(x-i^2)\\
&=\prod_{i=0}^{n-1}(x-i^2)(x-n^2)=\prod_{i=0}^{n}(x-i^2),\\
\end{split}
\end{equation*}
which completes the induction.
\end{proof}
\section{Log-Concavity and distribution of $|u_r(n,k)|$}
A sequence $c_0, c_1, \cdots, c_n$ of real numbers is unimodal if there exists  an index $ k $, $0 \leq k \leq n$ such that
 \[
 c_0 \leq c_1 \leq \cdots \leq c_k , \: \text{and} \: \: c_k \geq c_{k+1} \geq \cdots \geq c_n.
 \]
A stronger property than unimodality is long-concavity. A sequence $\{c_i\}_{i=0}^n$ of positive real numbers is log-concave when
\begin{equation} \label{E:log}
  c^2_k \geq c_{k+1} c_{k-1}, \quad \text{for all} \: k=1, 2, \cdots, n-1.
\end{equation}
 The sequence is said to be strict long-concave if $ "\geq "$ is replaced by $" > "$ in (\ref{E:log}).

 \begin{proposition}\textnormal{\cite[page 16]{mik}}
 If the sequence $\{c_i\}_{i=0}^n$ of positive real numbers is log-concave, then it is also unimodal.
 \end{proposition}
The generating functions can be used to prove strict log-concavity of a sequence by using Newton's inequality given in
   \cite[page 52]{hardy59}: If the generating function $\sum_{i=0}^{n}c_i x^i$ has only real roots
   then
   \begin{equation}
     c^2_k \geq \frac{k+1}{k} \: \frac{n-k+1}{n-k}  c_{k+1} c_{k-1}, \quad k=1, \cdots, n-1 .
   \end{equation}

If the coefficient sequence $\{c_i\}_{i=0}^n$ is positive then it is strictly log-concave, and is therefore unimodal \cite{wilf}.

The unsigned $r$-central factorial numbers of even indices of the first kind is defined by
  \begin{equation*}
   \mathfrak{u}_r(n,k)=(-1)^{n-k}u_r(n,k)=|u_r(n,k)|.
  \end{equation*}

\begin{theorem}\label{T:stc}
 For any fixed positive integer $n$, the sequence $\{\mathfrak{u}_r(n,k)\}_{k=0}^n$ is
strictly log-concave (and thus unimodal).
   \end{theorem}
 \begin{proof}
  Replacing $x$ by $-x-r$ in (\ref{E:rcent1}) gives
  \begin{equation}
   \sum_{k=0}^{n}  \mathfrak{u}_r(n,k)x^k=\prod_{j=0}^{n-1} (x+r+j^2).
   \end{equation}
   Let $ A_{n,r}(x)=\sum_{k=0}^{n} \mathfrak{u}_r(n,k)x^k $ be the polynomial associated to the sequence
   $\{\mathfrak{u}_r(n,k)\}_{k=0}^n$,
 \begin{equation*}
  A_{n,r}(x)=  \prod_{j=0}^{n-1} (x+r+j^2)=(x+r)(x+r+1^2)\cdots (x+r+(n-1)^2),
 \end{equation*}
clearly, all the roots of $ A_{n,r}(x)$ are real and negative.
 \end{proof}
Previous theorem implies that the sequence $\{\mathfrak{u}_r(n,k)\}_{k=0}^n$ satisfies the inequalities
\begin{equation}
\left(\mathfrak{u}_r(n,k)\right)^2 \geq \frac{k+1}{k} \: \frac{n-k+1}{n-k}\: \mathfrak{u}_r(n,k-1) \mathfrak{u}_r(n,k+1), \quad k=1, \cdots, n-1.
\end{equation}
\begin{theorem}\label{T:dist}
The array $\mathfrak{u}_r(n,k)$ is a Poisson-binomially distributed.
\end{theorem}
\begin{proof}
  Let us define random variables $Y_n$, $n=1,2,\cdots$, such that
  \begin{equation}\label{E:pr}
   P(Y_n=k) =\frac{\mathfrak{u}_r(n,k)}{\sum_{k=0}^{n}\mathfrak{u}_r(n,k)} =
    \frac{\mathfrak{u}_r(n,k)}{\prod_{i=0}^{n-1} (1+r+i^2)},\quad k=0,1,\cdots n.
  \end{equation}
 The probability generating function of $Y_n$ is given by
  \begin{equation} \label{E:pd}
 \begin{split}
 E(s^{Y_n})&=\sum_{k=0}^{n} s^k P(Y_n=k) = \prod_{i=0}^{n-1} \frac{s+r+i^2}{1+r+i^2}. \\
 &=\prod_{i=1}^{n-1} \left(1-\frac{1}{1+r+i^2} +\frac{s}{1+r+i^2}\right). \\
  \end{split}
 \end{equation}
 Then $Y_n$  can be represented as a sum of independent zero-one Bernoulli random variables, $X_0, X_1, \cdots , X_{n-1}$
   with probabilities $p_i$ of succuss on $i$-th trial
  \begin{equation}
 p_i= P(X_i=1) =1-P(X_i=0) =\frac{1}{1+r+i^2},
 \end{equation}
and then the random variable $Y_n=\sum_{i=0}^{n-1}X_i$ follows the Poisson-binomial distribution (which is a generalization of the binomial distribution) with mean and variance given by
\begin{equation}
E(Y_n)=\sum_{i=0}^{n-1} p_i=\sum_{i=0}^{n-1}\frac{1}{1+r+i^2}
\end{equation}
\begin{equation}
Var(Y_n)=\sum_{i=0}^{n-1} p_i(1-p_i)=   \sum_{i=0}^{n-1}\frac{r+i^2 }{(1+r+i^2)^2}.
\end{equation}
Note from (\ref{E:pr}) that $P(Y_n = k)$ differs from  the array $\mathfrak{u}_r(n, k)$  only by a normalizing constant, and thus completely characterizes the distribution of $\mathfrak{u}_r(n, k)$.
\end{proof}
Using the same previous assumption, one can get alternative proof of Theorem \ref{T:stc} as follows:

Let $f_n(k)=P(Y_n=k)$ be the probability distribution function of $Y_n$ defined in (\ref{E:pr}), equation (\ref{E:pd}) can be rewritten in the form
  \begin{equation*} \sum_{k=0}^{n} f_n(k)s^k=\prod_{i=0}^{n-1}( 1-p_i+p_i s).
 \end{equation*}

 An inequality of Newton found in \cite[page 104]{hardy59} and \cite{sam} states that if $\{a_i\}_{i=0}^{n-1}$ are any non-zero real numbers (positive or negative) and if $\{b_i\}_{i=0}^n$ are defined by
 \begin{equation*}
   \sum_{k=0}^{n} \binom{n}{k}  b_k s^k=\prod_{i=0}^{n-1} (1+a_i s),
 \end{equation*}
 then
 \begin{equation}
 b_k^2 > b_{k-1}b_{k+1} \quad \text{for} \: k=1, 2, \cdots , n.
\end{equation}
Setting $ b_k=\tfrac{f_n(k)}{\binom{n}{k}}$, we obtain
\begin{equation}
\left(\frac{f_n(k)}{\binom{n}{k}} \right)^2 >  \left(\frac{f_n(k-1)}{\binom{n}{k-1}} \right)\left(\frac{f_n(k+1)}{\binom{n}{k+1}} \right).
 \end{equation}
So, we have the inequality
 \[
\left(f_n(k)\right)^2 >f_n(k-1)f_n(k+1),
 \]
 that is $f_n(k)$ is strictly long-concave. Since $ f_n(k)= \mathfrak{u}_r(n,k)\prod_{i=0}^{n-1}p_i$, and $\prod_{i=0}^{n-1}p_i$ is clearly strict log-concave, then $\mathfrak{u}_r(n,k)$ also.
\section{Identities of the $r$-central factorial numbers with even indices}

\begin{theorem} For fixed $n\geq 0$, the generating functions of the arrays $u_r(n,k)$ and $U_r(n,k)$, are given, respectively,  by
\begin{equation}\label{gfu1}
\sum_{k=0}^{n}(-1)^{k}\:u_r(n,n-k)t^k=\prod_{j=0}^{n-1}(1+(j^2+r)t).
\end{equation}
\begin{equation}\label{gfu2}
  \sum_{n=k}^{\infty}U_r(n,k)t^{n}=\frac{t^k}{\prod_{j=0}^{k}(1-(j^2+r)t)}, \quad k\geq 0.
\end{equation}
\end{theorem}
\begin{proof}
  From (\ref{E:rcent1}), we get
  \begin{equation*}
    \sum_{k=0}^{n}u_r(n,k)x^k=\prod_{j=0}^{n-1}(x-(j^2+r)),
  \end{equation*}
setting $t^{-1}$ in place of $x$, then multiplying both sides by $t^n$, we obtain
\begin{equation*}
 \sum_{k=0}^{n}u_r(n,k) t^{n-k} =\prod_{j=0}^{n-1}(1-(j^2+r)t),
\end{equation*}
replacing $t$ by $-t$,
\begin{equation*}
 \sum_{k=0}^{n}(-1)^{n-k}\:u_r(n,k) t^{n-k} =\prod_{j=0}^{n-1}(1+(j^2+r)t),
\end{equation*}
then replace $k$ by $n-k$.

For (\ref{gfu2}), let $U_r^{(k)}(t)=\sum_{n\geq k}U_r(n,k)t^n$, then the initial condition is given by
\begin{equation*}
  U_r^{(0)}(t)=\sum_{n\geq k}U_r(n,0)t^n=\sum_{n\geq k}(rt)^n=\frac{1}{1-rt}.
\end{equation*}
 Multiplying both sides of (\ref{re2}) by $t^n$ and summing over $n \geq k$ gives
\begin{equation*}
  U_r^{(k)}(t)=t  U_r^{(k-1)}(t)+(k^2+r)t U_r^{(k)}(t).
\end{equation*}
Then
\begin{equation}
  U_r^{(k)}(t)=\frac{t}{1-(k^2+r)t} \:  U_r^{(k-1)}(t),\quad k\geq 1.
\end{equation}
Iterating this recurrence gives
\begin{equation*}
  U_r^{(k)}(t)=U_r^{(0)}(t)\frac{t}{1-(1+r)t}\:\frac{t}{1-(2^2+r)t}\:\cdots \frac{t}{1-(k^2+r)t}.
\end{equation*}
\end{proof}
Given a set of variables $z_1, z_2, \cdots , z_n$, the $k$-th elementary symmetric
function $\sigma_k(z_1 , z_2 , \ldots , z_n)$ and
the k-th complete homogeneous symmetric function $h_k(z_1 , z_2 , \ldots , z_n)$ are given, respectively, by
\[
 \sigma_k(z_1 , z_2 , \ldots , z_n)=\sum_{1 \leq j_1 < j_2 < \cdots < j_k \leq n}\:  z_{j_1}z_{j_2}\cdots z_{j_k}, \quad 1 \leq k\leq n
 \]
 \[
 h_k(z_1 , z_2 , \ldots , z_n)=\sum_{1\leq j_1 \leq j_2\leq \cdots\leq j_k \leq n}\:  z_{j_1}x_{z_2}\cdots z_{j_k}, \quad 1 \leq k\leq n
 \]
 by convention, we set $ \sigma_0(z_1 , z_2 , \ldots , z_n)=h_0(z_1 , z_2 , \ldots , z_n)=1$.
For $k > n$ or $k < 0$, we set $ \sigma_k(z_1 , z_2 , \ldots , z_n)=0$ and $ h_k(z_1 , z_2 , \ldots , z_n)=0$.

The generating functions of $\sigma_k $ and $h_k$ are given by
\begin{equation}
  \sum_{k\geq 0} \sigma_k(z_1 , z_2 , \ldots , z_n)t^k=\prod_{i=1}^{n}(1+z_it).
\end{equation}
\begin{equation}
 \sum_{k\geq 0} h_k(z_1 , z_2 , \ldots , z_n)t^k=\prod_{i=1}^{n}(1-z_it)^{-1}.
\end{equation}
Merca \cite{merca12} showed that
\begin{equation}\label{mer12}
 \sigma_i(z_1^2,\cdots,z_n^2)=\sum_{j=-i}^{i}(-1)^{j}\sigma_{i+j}(z_1,\cdots ,z_n)\sigma_{i-j}(z_1, \cdots, z_n).
\end{equation}

From (\ref{gfu1}) and (\ref{gfu2}), we deduce  that the numbers $u_r(n,k)$ and $U_r(n,k)$ are the elementary and complete symmetric functions:
\begin{equation}\label{E:srect1}
   u_r(n,n-k)=(-1)^{k} \sigma_{k}(r, 1^2+r, 2^2+r, \cdots , (n-1)^2+r).
\end{equation}
\begin{equation}\label{E:srect2}
   U_r(n+k,n)= h_{k}(r, 1^2+r, 2^2+r, \cdots , n^2+r).
  \end{equation}
At $r=0$, the central factorial numbers with even indices of both kinds satisfy
\begin{equation}
   u(n,n-k)=(-1)^{k} \sigma_{k}( 1^2, 2^2, \cdots , (n-1)^2).
\end{equation}
\begin{equation}
   U(n+k,n)= h_{k}(1^2, 2^2, \cdots , n^2).
  \end{equation}

\begin{theorem}\textnormal{(Merca \cite{merca13}).} Let $k$ and $n$ be two positive integers, then
\begin{equation}\label{mer13}
  g_k(z_1+t,z_2+t,\cdots , z_n+t)=\sum_{i=0}^{k}\binom{n-c_i}{k-i}g_i(z_1,z_2,\cdots,z_n)t^{k-i},
\end{equation}
where $t, z_1, z_2, \cdots, z_n$ are variables, $g$ is any of these complete or
elementary symmetric functions and
\begin{equation*}
  c_i=\begin{cases}
        i, & \mbox{if } g_i=\sigma_i, \\
        1-k, & \mbox{if} \:g_i=h_i.
      \end{cases}
\end{equation*}
\end{theorem}
In the next theorem, we show that the array $u_r(n,k)$ can be expressed in terms of $u(n,k)$ and vice versa.
\begin{theorem} Let $r$ be a positive integer, then
\begin{equation}\label{ev1}
  u_r(n,k)=\sum_{i=k}^{n}\binom{i}{k}\:(-r)^{i-k}u(n,i).
\end{equation}
\begin{equation}\label{ev11}
  u(n,k)=\sum_{i=k}^{n}\binom{i}{k}\:r^{i-k} \:u_r(n,i).
\end{equation}
\end{theorem}
\begin{proof}
Using (\ref{E:srect1}) and (\ref{mer13}), we obtain
\begin{equation*}
\begin{split}
  u_r(n,n-k)&=(-1)^{k} \sigma_{k}(r, 1^2+r, 2^2+r, \cdots , (n-1)^2+r)\\
  &=(-1)^k \sum_{i=0}^{k}\binom{n-i}{k-i}\sigma_i(0,1^2,\cdots,(n-1)^2)r^{k-i}\\
&= \sum_{i=0}^{k}\binom{n-i}{k-i}(-1)^i\;\sigma_i(0,1^2,\cdots,(n-1)^2)(-r)^{k-i}\\
&=\sum_{i=0}^{k}\binom{n-i}{k-i}u(n,n-i)(-r)^{k-i}.\\
\end{split}
\end{equation*}
Then replace $k$ by $n-k$ and $n-i$ by $i$ in that order.

For (\ref{ev11}), note that
\begin{equation*}
  \begin{split}
     u(n,n-k)&=(-1)^k \sigma_k(0,1^2,\cdots, (n-1)^2)  \\
       & =(-1)^k\sum_{i=0}^{k}\binom{n-i}{k-i} \sigma_i(r,1^2+r,\cdots, (n-1)^2+r)(-r)^{k-i}\\
       & =\sum_{i=0}^{k}\binom{n-i}{k-i}(r)^{k-i}u_r(n,n-i).
  \end{split}
\end{equation*}
\end{proof}
The $r$-central numbers with even indices of the first kind, $u_r(n,k)$ may be also
calculated by Stirling numbers $s(n, k)$ of the first kind as follows.
\begin{theorem}If $n,k,r \geq 0$, then

\begin{equation}\label{stir}
   u_r(n,n-k)= \sum_{i=0}^{k} \sum_{j=-i}^{i}(-1)^{k+j} \binom{n-i}{k-i}s(n, n-i+j) s(n, n-i-j)r^{k-i}.
\end{equation}
\end{theorem}
\begin{proof} From (\ref{mer12}), (\ref{E:srect1}) and (\ref{mer13}), we get
\begin{equation*}
\begin{split}
  u_r(n,n-k)&=(-1)^{k} \sigma_{k}(r, 1^2+r, 2^2+r, \cdots , (n-1)^2+r)\\
  &=(-1)^k \sum_{i=0}^{k}\binom{n-i}{k-i}\sigma_i(0,1^2,\cdots,(n-1)^2)r^{k-i}\\
&= \sum_{i=0}^{k} \sum_{j=-i}^{i} (-1)^{k+j} \binom{n-i}{k-i}\sigma_{i+j}(0,1,\cdots ,n-1)\sigma_{i-j}(0, 1, 2, \cdots, n-1)r^{k-i}\\
&= \sum_{i=0}^{k} \sum_{j=-i}^{i}(-1)^{k+j} \binom{n-i}{k-i}s(n, n-i+j) s(n, n-i-j)r^{k-i}.\\
\end{split}
\end{equation*}
 \end{proof}
\begin{remark}
From (\ref{stir}), we note that at $i=k$, we get the central factorial numbers with even indices of the first kind, $u(n,k)$ in terms of Stirling numbers of the first kind  (see \cite{merca12})
 \begin{equation*}
    u(n,n-k)=  \sum_{j=-k}^{k}(-1)^{k+j} s(n, n-k+j) s(n, n-k-j).
 \end{equation*}
 \end{remark}
\begin{example} For example
\begin{equation*}
  u_r(5,3)= \sum_{i=0}^{2} \sum_{j=-i}^{i}(-1)^{2+j} \binom{5-i}{2-i}s(5, 5-i+j) s(5, 5-i-j)r^{2-i}=10r^2+120 r+273.
\end{equation*}
\begin{equation*}
  u_r(6,5)= \sum_{i=0}^{1} \sum_{j=-i}^{i}(-1)^{1+j} \binom{6-i}{1-i}s(6, 6-i+j) s(6, 6-i-j)r^{1-i}=-6r-55.
\end{equation*}
\end{example}
\begin{remark}
We can give alternative proof of (\ref{ev2}) by using (\ref{E:srect2}) and (\ref{mer13}) as follows
\begin{equation*}
\begin{split}
  U_r(n+k,n)&=h_k(r,1^2+r, \cdots, n^2+r)\\
  &=\sum_{i=0}^{k}\binom{n+k}{k-i}h_i(0^2, 1^2, \cdots, n^2)r^{k-i}\\
  &=\sum_{i=0}^{k}\binom{n+k}{k-i}U(n+i,n)r^{k-i}.
  \end{split}
\end{equation*}
and then replace $n$ by $n-k$, $k$ by $n-k$, and $i$ by $i-k$ in that order.
\end{remark}
The array $U_r(n.k)$ expressed in terms of $U(n,k)$ by (\ref{ev2}), here we show that the array $U(n,k)$ can also be expressed in terms of $U_r(n.k)$.
\begin{theorem} For $n, k, r \geq 0$,
\begin{equation}\label{cr2}
  U(n,k)=\sum_{i=k}^{n}\binom{n}{i}(-r)^{n-i}U_r(i,k)
\end{equation}
\end{theorem}
\begin{proof}From (\ref{E:srect2}) and (\ref{mer13}), we have
  \begin{equation*}
    \begin{split}
       U(n+k,n)&=h_k(0,1^2,\cdots,n^2)  \\
         & =\sum_{i=0}^{k} \binom{n+k}{k-i}h_i(r,1^2+r,\cdots, n^2+r)(-r)^{k-i}\\
         & =\sum_{i=0}^{k} \binom{n+k}{k-i}(-r)^{k-i}U_r(n+i,n).
    \end{split}
  \end{equation*}
\end{proof}
\section{The $r$-central factorial matrices }
In the following, we introduce the $r$-central factorial matrices with even indices of both kinds, then we obtain  factorization of these matrices.

From (\ref{E:rcent1})and (\ref{E:rcent2}), we get the following orthogonal relation
\begin{equation}\label{orth}
  \sum_{k=i}^{n}u_r(n,k)\:U_r(k,i)=\sum_{k=i}^{n}U_r(n,k)\:u_r(k,i)=\delta_{ni}.
\end{equation}
An orthogonal relation (\ref{orth}) entails a pair of inverse relations
\begin{equation*}
  a_n=\sum_{k=0}^n U_r(n,k)\:b_k,
\end{equation*}
\begin{equation*}
   b_n=\sum_{k=0}^n u_r(n,k)\:a_k.
\end{equation*}
Inverse relations are known for practically all special numbers, such as Whitney numbers and Stirling numbers (see for example \cite{{beih16}, {qalpha},  mezo15, {riod}}).
\begin{definition}
The $r$-central factorial matrices with even indices of the first kind $\mathcal{U}_1(n)$, and of the second kind $\mathcal{U}_2(n)$ are the $n\times n$  matrices defined by
\begin{equation*}
\mathcal{U}_1(n):=\mathcal{U}_1^{(r)}(n)=[u_r(i,j)]_{0\leq i, j \leq n-1},
\end{equation*}
and
\begin{equation*}
\mathcal{U}_2(n):=\mathcal{U}_2^{(r)}(n)=[U_r(i,j)]_{0\leq i, j \leq n-1}.
\end{equation*}

\end{definition}
For example, $\mathcal{U}_1(5)$ is given by
{\footnotesize
\[
\left[
\begin{matrix}
  1                     & 0                    &0               &0           & 0         \\
  -r                    & 1                    & 0              & 0          & 0         \\
  r(1+r)                & -2r-1                & 1              & 0          & 0         \\
  -r(1+r)(2^2+r)        & 3r^2+10r+4           & -3r-5          & 1          & 0         \\
  r(1+r)(2^2+r)(3^2+r)  & -4r^3-42r^2-98r-36   & 6r^2+42r+49    & -4r-14     & 1
\end{matrix}
\right],
\]}
and $\mathcal{U}_2(5)$ is given by
\[
\left[
\begin{matrix}
1       & 0               & 0             &0          & 0         \\
r      & 1                & 0             & 0         & 0         \\
r^2    & 2r+1             & 1             & 0         & 0         \\
r^3    &3r^2+3r+1         & 3r+5          & 1         & 0         \\
r^4    &4r^3+6r^2+4r+1    & 6r^2+20r+21   & 4r+14     & 1
\end{matrix}
\right].
\]
In particular, setting $r=0$, we get the central factorial matrices with even indices of both kinds
\begin{equation*}
\mathcal{A}_1(n)=[u(i,j)]_{0\leq i, j \leq n-1}, \quad \textrm{and}  \quad \mathcal{A}_2(n)=[U(i,j)]_{0\leq i, j \leq n-1}.
\end{equation*}
The orthogonality property (\ref{orth}) is equivalent to the matrix equation
\begin{equation*}
  \mathcal{U}_1(n) \: \mathcal{U}_2(n)=\mathcal{U}_2(n) \:\mathcal{U}_1(n)=\mathbf{I},
\end{equation*}
with $\mathbf{I}$ is the $n\times n$ unit matrix. Hence,
\begin{equation*}
(\mathcal{U}_1(n))^{-1} =\mathcal{U}_2(n) \quad \text{and}\quad (\mathcal{U}_2(n))^{-1} =\mathcal{U}_1(n)
\end{equation*}
The generalized $n\times n$ Pascal matrix $\textbf{P}_n[z]$ is defined as follows (see \cite{call}):
\begin{equation}\label{pascal}
 \textbf{P}_n[z]=\left[\binom{i}{j} z^{i-j}\right]_{0\leq i, j \leq n-1},
\end{equation}
with $\textbf{P}_n=\textbf{P}_n[1]$, the Pascal matrix $\textbf{P}_n$ of order $n$.
For example
\[
\textbf{P}_5[z]=
\left[
\begin{matrix}
1       & 0           & 0        &0          & 0         \\
z      & 1            & 0        & 0         & 0         \\
z^2    & 2z           & 1        & 0         & 0         \\
z^3    &3z^2          & 3z       & 1         & 0         \\
z^4    &4z^3          & 6z^2     & 4z        & 1
\end{matrix}
\right]
\]
Moreover,
\[
\textbf{P}_n^{-1}[z]=\textbf{P}_n[-z]=\left[(-1)^{i-j}\binom{i}{j}z^{i-j}\right]_{0\leq i, j \leq n-1}
\]
From (\ref{ev1}) and (\ref{ev2}), we have the following factorization
\begin{equation}
  \mathcal{U}_1(n)=\mathcal{A}_1(n)P_n[-r],\quad n \geq 1,
\end{equation}
and
\begin{equation}
  \mathcal{U}_2(n)= P_n[r]\mathcal{A}_2(n),\quad n \geq 1.
\end{equation}
For example
\[
\begin{split}
\mathcal{U}_1(5)&=
\left[
\begin{matrix}
  1     & 0    &0      &0        & 0         \\
  0     & 1    & 0     & 0       & 0         \\
  0     & -1   & 1     & 0       & 0         \\
  0     & 4    &-5     & 1       & 0         \\
  0     &-36   & 49    & -14     & 1
\end{matrix}
\right]
\times
\left[
\begin{matrix}
1       & 0           & 0        &0          & 0         \\
-r      & 1            & 0        & 0         & 0         \\
r^2    & -2r           & 1        & 0         & 0         \\
-r^3    &3r^2          & -3r       & 1         & 0         \\
r^4    &-4r^3          & 6r^2     & -4r        & 1
\end{matrix}
\right]\\
&=\mathcal{A}_1(5)\textbf{P}_5[-r].
\end{split}
\]
and
\[
\begin{split}
\mathcal{U}_2(5)&=
\left[
\begin{matrix}
1       & 0           & 0        &0          & 0         \\
r      & 1            & 0        & 0         & 0         \\
r^2    & 2r           & 1        & 0         & 0         \\
r^3    &3r^2          & 3r       & 1         & 0         \\
r^4    &4r^3          & 6r^2     & 4r        & 1
\end{matrix}
\right]
\times
\left[
\begin{matrix}
  1     & 0    &0      &0        & 0         \\
  0     & 1    & 0     & 0       & 0         \\
  0     & 1   & 1     & 0       & 0         \\
  0     & 1    &5     & 1       & 0         \\
  0     &1   & 21    & 14     & 1
\end{matrix}
\right]\\
&=\textbf{P}_5[r]\mathcal{A}_2(5).
\end{split}
\]


\begin{thebibliography}{99}
\bibitem{but}P. L. Butzer, K. Schmidt, E.L. Stark and L. Vogt,  Central factorial numbers, their
   main properties and some applications, Numer. Funct. Anal. Optim., \textbf{10(5\&6)}(1989), 419--488.
\bibitem{call} G. S. Call and D. J. Velleman, Pascal's matrices, Amer. Math. Monthly \textbf{100} (1993), 372--376.
\bibitem{beih16} B. S. El-Desouky, Nenad P. Caki\'c, and F. A. Shiha, New Family of Whitney Numbers, Filomat \textbf{31(2)} (2017), 309--320.
\bibitem{qalpha} B. S. El-Desouky and F. A. Shiha  A q-analogue of $\bar{\alpha}$-Whitney Numbers, Appl. Anal. Discrete Math., \textbf{12} (2018), 178--191.
\bibitem{hardy59} Hardy, G.H., Littlewood,J. E. and Po'lya, G. Inequalities. Cambridge Univ., 1959.
\bibitem{gel}Y. Gelineau, J. Zeng, Combinatorial interpretations of the Jacobi-Stirling
    numbers, Electron. J. Combin., \textbf{17}(2010), Article Number: R70.
\bibitem{kim119} D.S. Kim, D.V. Dolgy, D. Kim, T. Kim, Some identities on $r$-central factorial numbers and
     $r$-central Bell polynomials, Adv. Difference Equ. \textbf{2019}, Article number: 245(2019).
\bibitem{kim19} T. Kim, D.S. Kim, G.-W. Jang, J. Kwon,  Extended central factorial polynomials of the second kind.
     Adv. Difference Equ. \textbf{2019}, Article number: 24(2019)
\bibitem{mansour12} T. Mansour, S. Mulay, M. Shattuck,   A general two-term recurrence and
    its solution. European. J. Combin., \textbf{33}(2012) 20--26.
\bibitem{merca12} M. Merca, A Special Case of the Generalized Girard-Waring Formula, J. Integer Seq., \textbf{15} (2012), Article 12.5.7.
\bibitem{merca13} M. Merca, A note on the r-Whitney numbers of Dowling lattices, C. R. Math. Acad. Sci. Paris, Ser. I \textbf{351} (2013), 649--655.
\bibitem{merca2016} M. Merca, Connections between central factorial numbers and Bernoulli polynomials,
    Period. Math. Hungar., \textbf{73(2)}(2016), 259--264.
\bibitem{mezo15}  I. Mez\"o and J. L.  Ram\'irez, The linear algebra of the r-Whitney matrices, Integral
Transforms Spec. Funct.\textbf{26(3)} (2015), 213--225.
\bibitem{mik}  Mikl\'os B\'ona, Combinatorics of Permutations, Chapman \& Hall/CRC (2004).
\bibitem{riod}J. Riordan,  Combinatorial Identities; JohnWiley \& Sons, Inc.: New York, NY, USA, 1968.
\bibitem{sam} S. M. Samuels, On the number of successes in independent trials, Ann. Math. Statist., \textbf{36(4)} (1965), 1272--1278.
 \bibitem{wilf} H.S. Wilf, Generating Functionology, Academic Press/Harcourt Brace Jovanovich, 1994.
 \end{thebibliography}
\end{document}